\documentclass{amsart}
\usepackage[margin=1.5cm]{geometry}
\numberwithin{equation}{section}
\usepackage{amssymb}
\usepackage{amsmath, amsfonts,amsthm,amssymb,amscd, verbatim,graphicx,color,multirow,booktabs, caption,tikz,tikz-cd, mathdots,bm}
\usepackage{tikz-cd}
\usetikzlibrary{positioning}
\newtheorem{theorem}{Theorem}[section]
\newtheorem{corollary}[theorem]{Corollary}
\newtheorem{lemma}[theorem]{Lemma}

\newcommand{\Sym}{\mathop{\mathrm{Sym}}}
\newcommand{\aut}{\mathop{\mathrm{Aut}}}
\newcommand{\soc}{\mathop{\mathrm{soc}}}
\newcommand{\perm}{\mathop{\mathrm{Sym}}}

\def\nor#1#2{{\bf N}_{{#1}}({{#2}})}

\def\cent#1#2{{\bf C}_{{#1}}({{#2}})}

\begin{document}

\title[$\mathcal{B}_{pp}$-groups]{Independent sets of generators of prime power order}
\author[A. Lucchini]{Andrea Lucchini}
\address{Andrea Lucchini, Dipartimento di Matematica \lq\lq Tullio Levi-Civita\rq\rq,\newline
 University of Padova, Via Trieste 63, 35121 Padova, Italy} 
\email{lucchini@math.unipd.it}
         
\author[P. Spiga]{Pablo Spiga}
\address{Pablo Spiga, Dipartimento di Matematica Pura e Applicata,\newline
 University of Milano-Bicocca, Via Cozzi 55, 20126 Milano, Italy} 
\email{pablo.spiga@unimib.it}
	\maketitle

        \begin{abstract}
A subset $X$ of
a finite group $G$ is said to be prime-power-independent if each element in $X$ has prime power order and there is no proper subset $Y$ of $X$ with $\langle Y, \Phi(G)\rangle = \langle X, \Phi(G)\rangle$, where $\Phi(G)$ is the Frattini subgroup of $G$. A group $G$ is $\mathcal{B}_{pp}$ if all prime-power-independent
generating sets for $G$ have the same cardinality. We prove that, if $G$ is $\mathcal{B}_{pp}$, then $G$ is solvable. Pivoting on some recent results of Krempa and Stocka~\cite{ks,s}, this yields a complete classification of $\mathcal{B}_{pp}$-groups.         
          \end{abstract}
\section{Introduction}\label{sec:introduction}

Throughout this paper, all groups are finite. We start this introductory section with some definitions fundamental for our work.
Given a group $G$, an element $g\in G$ is said to be a $pp$-\textit{\textbf{element}} if $g$ has prime power order. A subset $X$ of $G$ is said to be
\begin{description}
\item[\textit{\textbf{independent}}] if $\langle X,\Phi(G)\rangle\ne \langle Y,\Phi(G)\rangle$ for every proper subset $Y$ of $X$ (where as customary we denote by $\Phi(G)$ the \textit{\textbf{Frattini subgroup}} of $G$);
\item[$pp$-\textit{\textbf{independent}}] if $X$ is independent and each element in $X$ is a $pp$-element; and
\item[$pp$-\textit{\textbf{base}}]if $X$ is a $pp$-independent generating set for $G$.
\end{description}
 Finally, $G$ is said to be a $\mathcal{B}_{pp}$-\textit{\textbf{group}} if every two $pp$-bases of $G$ have the same cardinality.

The main result of this paper is the following.
\begin{theorem}
\label{thrm:main}
If $G$ is a $\mathcal{B}_{pp}$-group, then $G$ is solvable. 
\end{theorem}

Theorem~\ref{thrm:main} gives a solution to Question~1 in~\cite{ks} in a strong sense. In fact,  it yields a complete classification of the $\mathcal{B}_{pp}$-groups. Indeed, Krempa and Stocka~\cite{ks,s} have obtained an entirely satisfactory classification of solvable $\mathcal{B}_{pp}$-groups and hence Theorem~\ref{thrm:main} together with the work in~\cite{ks,s} gives a classification of all $\mathcal{B}_{pp}$-groups. This classification is easier to formulate for Frattini-free groups, that is, for groups $G$ with $\Phi(G)=1$. (Observe that $G$ is a $\mathcal{B}_{pp}$-group if and only if so is $G/\Phi(G)$.)
\begin{corollary}\label{corollary}
Let $G$ be a group with $\Phi(G)=1$. Then $G$ is a $\mathcal{B}_{pp}$-group and if only if one of the following holds:
\begin{enumerate}
\item\label{item1}$G$ is an elementary abelian $p$-group,
\item\label{item2}$G= P\rtimes Q$,
where $P$ is an elementary abelian $p$-group, $Q$ is a non-identity cyclic $q$-group
for distinct prime numbers $p$ and  $q$ such that $Q$ acts faithfully on $P$ and the $(\mathbb{Z}/p\mathbb{Z})[Q]$-module $P$ is a direct sum of pair-wise isomorphic simple modules,
\item\label{item3}$G$ is a direct product of groups given in~\eqref{item1} or in~\eqref{item2} with pair-wise coprime orders.
\end{enumerate} 
\end{corollary}
The groups as in~\eqref{item2} are simply refereed to as scalar extensions in~\cite{s}. We refer the reader to the work of Krempa and Stocka~\cite{ks,s} for various motivations on investigating $\mathcal{B}_{pp}$-groups. Broadly speaking, this motivation is rooted on independent generating sets and on generalizations of the Burnside basis theorem; in turn, these motivations are useful for studying groups satisfying the exchange property for bases which is useful for constructing matroids starting from finite groups. 

As a bi-product of the arguments used in the proof of Theorem 
\ref{thrm:main}, we obtain the following result of independent interest. (See Section~\ref{notation} for undefined terminology.)

\begin{theorem}\label{bound}Let $G$ be a group and denote by $m(G)$ the largest cardinality of an independent generating set of $G$. Then $m(G)\geq a+b,$ where $a$ and $b$ are, respectively, the number of non-Frattini and non-abelian factors in a chief series of $G$.
\end{theorem}

We have verified with a computer computation~\cite{magma} that the bound in Theorem~\ref{bound} is sharp when 
$G$ is the automorphism group of the alternating group of degree $6$: here, $m(G)=4$, $a=3$ and $b=1$. Theorem~\ref{bound} gives a strengthening of  the bound $m(G)\ge a$, which was proved in~\cite{l1}. Here, it was also proved that $m(G)=a$ for every solvable group.

The structure of the paper is straightforward. In Section~\ref{preliminaries} after establishing some notation, and after a short detour through fixed point ratios and spreads, we give some basic results. In Section~\ref{proofs} after establishing a few rather technical results, we prove Theorem~\ref{thrm:main} and Corollary~\ref{corollary}. Finally, we prove Theorem~\ref{bound} in Section~\ref{sec:3}.

\section{Preliminaries}\label{preliminaries}
\subsection{Notation}\label{notation}
Given a group $G$, we let $m(G)$ and $m_{pp}(G)$ denote the largest cardinality of an independent generating set of $G$ and of a $pp$-independent generating set for $G$. Since every $pp$-independent generating set is also an independent generating set, we have $m(G)\ge m_{pp}(G)$. In fact, in Lemma~\ref{l:1} we show that $m(G)=m_{pp}(G)$.

Let $$1=G_t \unlhd \dots \unlhd G_0=G$$
be a chief series for $G$. A factor $G_i/G_{i+1}$ is said to be a \textit{\textbf{non-abelian}} chief factor of $G$ if $G_i/G_{i+1}$ is a non-abelian group; moreover, $G_i/G_{i+1}$ is said to be a \textit{\textbf{Frattini}} chief factor of $G$ if $G_i/G_{i+1}\le \Phi(G/G_{i+1})$. 

The \textit{\textbf{socle}} of $G$, denoted by $\soc G$, is the subgroup generated by the minimal normal subgroups of $G$. In particular, if $\soc G$ is a minimal normal subgroup of $G$ (that is, $G$ has a unique minimal normal subgroup), then $G$ is said to be \textit{\textbf{monolithic}}.

Let $G$ be a monolithic group with socle $N$. Following the notation in~\cite{l2}, we define $\mu(G):=m(G)-m(G/N)$.

Given a positive integer $n$ and a group $H$, we denote by $H\mathrm{wr} \Sym(n)$ the \textit{\textbf{wreath product}} of $H$ with the symmetric group $\Sym(n)$ of degree $n$. We denote the elements of $H\mathrm{wr} \Sym(n)$ with ordered pairs $f\sigma$, where $f\in H^n$ and $\sigma\in \Sym(n)$.

Given two  positive integers $x$ and $n$ with $x,n\ge 2$, we say that the prime $r$ is a \textit{\textbf{primitive prime divisor}} of $x^n-1$ if $r$ divides $x^n-1$ and $r$ is relatively prime to $x^i-1$, for each $i\in \{1,\ldots,n-1\}$. From a celebrated theorem of Zsigmondy~\cite{z}, either $x^n-1$ has a primitive prime divisor, or $n=6$ and $x=2$, or $n=2$ and $x+1$ is a power of $2$. In the latter case, when $x$ is a prime power, we deduce that $x$ must be a (Mersenne) prime. We actually need the following refinement. The prime $r$ is said to be a \textit{\textbf{large primitive prime divisor}} of $x^n-1$ if $r$ is a primitive prime divisor of $x^n-1$ and either $r>n+1$ or $r^2$ divides $x^n-1$. We recall the classical result of Feit~\cite{Feit} on the existence of large primitive prime divisors. (We refer also to~\cite{Feit1}, for an elementary proof of this result.)
\begin{lemma}\label{feit}
If $x$ and $n$ are integers greater than $1$ there there exists a large primitive prime divisor for $x^n-1$ except exactly in the following cases:
\begin{enumerate}
\item $n=2$ and $x=2^s3^t-1$ for some natural numbers $s\ge 0$ and $t\in \{0,1\}$ with $s\ge 2$ if $t=0$,
\item $x=2$ and $n\in \{4,6,10,12,18\}$,
\item $x=3$ and $n\in \{4,6\}$,
\item $x=5$ and $n=6$.
\end{enumerate}
\end{lemma}

Our last two definitions are rather technical and (for our application) they only pertain to almost simple groups, but they will prove useful. Given an almost simple group $H$ with socle $S$ and a subgroup $K$ of $H$ with $H=KS$, let $$t(H,K)$$ be the smallest cardinality of a set $X$ of $pp$-elements in $S$ with $H=\langle K, X\rangle$. Then, define $$t(H):=\max\{t(H,K)\mid K\le H \hbox{ with } H=KS\}.$$ 
From~\cite[Theorem~1]{k}, $S$ is generated by an involution and by an element of odd prime power order and hence
\begin{equation}
\label{eq:silly2}
t(H)\le 2.
\end{equation}

Given a subgroup $K$ of $H$, we say that a subset $Y$ of $H$ is $K$-\textit{\textbf{generating}} for $H$ if $H=\langle K,Y\rangle$. A $K$-generating set for $H$ is said to be $K$-\textit{\textbf{independent}} if no proper subset of $Y$ generates $H$ together with $K$. We denote by $$m_K(H)$$ the largest cardinality of a $K$-independent generating set for $H$.

\subsection{A (short) walk through fixed point ratios and spreads}\label{fpr}
Let $H$ be an almost simple group with socle $S$ and let $g,s\in H$. We set
$$P(g,s):=\frac{|\{t\in s^H\mid \langle g,t\rangle\not\ge S\}|}{|s^H|}.$$
This definition is strictly related to the definition of spread and uniform spread in almost simple groups and we refer the reader to~\cite{bgk,GK} for further details.

For any action of $H$ on a set $\Omega$ and for any $g\in H$, consider the set $\mathrm{Fix}_\Omega(g):=\{\omega\in \Omega\mid \omega^g=\omega\}$ of fixed points of $g$ on $\Omega$ and the \textit{\textbf{fixed point ratio}}
$$\mu(g,\Omega):=\frac{|\mathrm{Fix}_\Omega(g)|}{|\Omega|}.$$
From~\cite[Section~$2$]{GK}, if $M\backslash H$ denotes the set of right cosets of the subgroup $M$ of $H$, then
\begin{equation}\label{eq:vvv}\mu(g,M\backslash H)=\frac{|g^H\cap M|}{|g^H|}.
\end{equation}
Let now $\mathcal{M}(H,g)$ be the collection of all maximal subgroups of $H$ containing $g$ and assume that $H$ is almost simple with socle $S$. Then, from~\eqref{eq:vvv}, we deduce
\begin{equation}\label{eq:spread}
P(g,s)\le \sum_{M\in \mathcal{M}(H,g)}\frac{|\{t\in
 s^H\mid \langle g,t\rangle\le M\}|}{|s^H|}=\sum_{M\in \mathcal{M}(H,s)}\frac{|\{h\in
 g^H\mid \langle h,s\rangle\le M\}|}{|g^H|}\le\sum_{M\in\mathcal{M}(H,s)}\mu(g,M\backslash H).
\end{equation}
Eq.~\eqref{eq:spread} also appears in~\cite[(2.4)]{bgk}.
We summarize in the following lemma the main application of fixed point ratios in our context.
\begin{lemma}\label{l:spread}
Let $H$ be an almost simple group with socle $S$. Suppose $H\ne S$. If, for every $g\in H\setminus S$,  there exists a $pp$-element $s_g\in S$ with  $P(g,s_g)<1$, then $t(H)=1$. In particular, if  $\sum_{M\in\mathcal{M}(H,s)}\mu(g,M\backslash H)<1$ for every $g\in H\setminus S$, then $t(H)=1$.
\end{lemma}
\begin{proof}
Let $K$ be a subgroup of $H$ with $H=KS$. For every $g\in K\setminus S$, let $s_g$ be a $pp$-element belonging to $S$ with $P(g,s_g)<1$. Then by definition of $P(g,s_g)$, there exists $t\in s_g^H$ with $\langle g,t\rangle\ge S$. Thus $H=\langle K,t\rangle$ and hence $t(H,K)=1$. Since this holds regardless of $K$, we have $t(H)=1$. The rest of the proof follows from~\eqref{eq:spread}.
\end{proof}


\subsection{Basic results}\label{basicresults}

\begin{lemma}\label{l:1}
Let $G$ be a group. Then $m(G)=m_{pp}(G)$.
\end{lemma}
\begin{proof}
As we have observed above, $m(G)\ge m_{pp}(G)$ and hence we only need to show that $m(G)\le m_{pp}(G)$.

Let $X:=\{x_1,\dots,x_{m(G)}\}$ be an independent generating set for $G$ of cardinality $m(G)$. For each $i\in \{1,\ldots,m(G)\}$, we may write $x_i=y_{1,i}\cdots y_{k_i,i}$, where 
$y_{1,i},\ldots,y_{k_i,i}$ are pair-wise commuting $pp$-elements of $G$ with 
\begin{equation}\label{eq:silly}\langle x_i\rangle=\langle y_{1,i},\ldots,y_{k_i,i}\rangle.
\end{equation} Clearly, $$\{y_{j,i}\mid 1\leq j\leq  k_i, 1\leq i \leq m(G)\}$$ is a generating set for $G$ consisting of $pp$-elements and  hence it contains a $pp$-base $Y$.

We claim that, for each $i\in \{1,\ldots,m(G)\}$, there exists $j\in \{1,\ldots,k_i\}$ with $y_{j,i}\in Y$. Indeed, if for some some $\bar{i}$, $Y$ contains no $y_{j,\bar{i}}$, then
$$G=\langle Y\rangle\le\langle y_{j,i}\mid i\in \{1,\ldots,m(G)\}\setminus\{\bar{i}\},j\in \{1,\ldots,k_i\}\rangle\le \langle X\setminus\{x_{\bar{i}}\}\rangle,$$
where in the last inequality we have used~\eqref{eq:silly}. However, this contradicts the fact that $X$ is independent and hence the claim is proved. 

The previous paragraph yields $|Y|\ge m(G)$ and hence the lemma follows because $m_{pp}(G)\ge |Y|$.
\end{proof} 

We now recall~\cite[Theorem~6.1~(1)]{ks}.
\begin{lemma}\label{induction}
If $G$ is a $\mathcal{B}_{pp}$-group, then every quotient of $G$ is a 
$\mathcal{B}_{pp}$-group.
\end{lemma}
\section{Proofs of Theorem~$\ref{thrm:main}$ and Corollary~$\ref{corollary}$}\label{proofs}
\subsection{Technical lemmas}
\begin{lemma}\label{technical1}Let $q$ be a prime power with $q\ge 4$ and let $H$ be an almost simple group with socle 
$S:=\mathrm{PSL}_2(q)$ and with $H\ne S$. Then $t(H)=1$.
\end{lemma}
\begin{proof}
It suffices to prove that, for every subgroup $K$ of $H$ with $H=KS$, there exists a $pp$-element $x_K\in S$ with $H=\langle K,x_K\rangle$. Write $q:=p^f$, where $p$ is a prime number and $f$ is a positive integer.

Let $K\le H$ with $H=KS$ and let $\theta\in K\setminus S$.   Assume that $p^{2f}-1$ admits no large primitive prime divisor.  From Lemma~\ref{feit}, we deduce that either 
$$S\in \{
\mathrm{PSL}_2(4)=\mathrm{PSL}_2(5),
\mathrm{PSL}_2(8),
\mathrm{PSL}_2(32),
\mathrm{PSL}_2(64),
\mathrm{PSL}_2(512),
\mathrm{PSL}_2(9),
\mathrm{PSL}_2(27),
\mathrm{PSL}_2(125)
\},$$ or $f=1$ and $q=p=2^s3^t-1$ for some natural numbers  $s\ge 0$ and $t\in \{0,1\}$ with $s\ge 2$ if $t=0$. 
In the first eight exceptional cases, the result can be established with a direct inspection using, for instance, the assistance of the computer algebra system~\texttt{magma}~\cite{magma}. 
We now consider the case $q=p=2^s3^t-1$. Actually, we deal with the more general case that $q=p$ is a prime number. As $H\ne S$, we have  $H=\mathrm{PGL}_2(q)$.
Clearly, a Sylow $p$-subgroup of $S$ is cyclic; let $x\in S$ be an element generating a Sylow $p$-subgroup of $S$. Observe that we may choose $x$ so that $\theta$ does not normalize $\langle x\rangle$. Using the list of the maximal subgroups of $S$ (see for instance~\cite[Tables~$8.1$,~$8.2$]{bray}), we see that $S=\langle x,x^\theta\rangle$. Thus $H=\langle K,x\rangle$ and $t(H,K)=1$. 

Assume now that $p^{2f}-1$ admits a large primitive prime divisor  $r$. Observe that, from the previous paragraph, we may suppose that $f>1$. In particular, either $r>2f+1\ge 5$ or $r^2$ divides $q+1$. Clearly, a Sylow $r$-subgroup of $S$ is cyclic; let $x\in S$ be an element generating a Sylow $r$-subgroup of $S$. Observe that we may choose $x$ so that $\theta$ does not normalize $\langle x\rangle$ (this can be easily established by considering the structure of the subgroup lattice of $S$, see~\cite[Table~$8.1$]{bray}). Using the list of the maximal subgroups of $S$ (see for instance~\cite[Tables~$8.1$,~$8.2$]{bray}), we see that either 
\begin{itemize}
\item $S=\langle x,x^\theta\rangle$, or 
\item $r=5$ and $\langle x,x^\theta\rangle \cong \mathrm{Alt}(5)$, or 
\item $r=3$ and $\langle x,x^\theta\rangle$ is isomorphic to either $\mathrm{Alt}(4)$ or $\mathrm{Alt}(5)$.
\end{itemize}
In the first case, $H=\langle K,x\rangle$ and hence $t(H,K)=1$. In the last two cases, $r$  is the cardinality of a Sylow $r$-subgroup of $S$, because $5$ is the cardinality of a Sylow $5$-subgroup of $\mathrm{Alt}(5)$ and $3$ is the cardinality of a Sylow $3$-subgroup of $\mathrm{Sym}(4)$. However, this contradicts the fact that $r$ is a large primitive prime divisor of $p^{2f}-1$.
\end{proof}

\begin{lemma}\label{technical2U}
Let $q$ be a prime power and let $H$ be an almost simple group with socle $S:=\mathrm{PSU}_3(q)$ and with $H\ne S$. Then $t(H)=1$.
\end{lemma}
\begin{proof}
As  $\mathrm{PSU}_3(2)$ is solvable, we have $q>2$. Here the argument is similar to the proof of Lemma~\ref{technical1}: we use primitive prime divisors and the structure of the subgroup lattice of $S$, see~\cite[Tables~$8.5$,~$8.6$]{bray}. Write $q:=p^f$, where $p$ is a prime number and $f$ is a positive integer.

Let $K\le H$ with $H=KS$ and let $\theta\in K\setminus S$.  Assume $p^{6f}-1$ admits a large primitive prime divisor of $r$. Clearly, a Sylow $r$-subgroup of $S$ is cyclic; let $x\in S$ be an element generating a Sylow $r$-subgroup of $S$. Observe that we may choose $x$ so that $\theta$ does not normalize $\langle x\rangle$. Using the list of the maximal subgroups of $S$ (see~\cite[Tables~$8.5$,~$8.6$]{bray}), we see that $S=\langle x,x^\theta\rangle$ (here we are using the fact that $r$ is a large Zsigmondy prime and hence $\langle x,x^\theta\rangle$ cannot be contained in a maximal subgroup in the Aschbacher class $\mathcal{S}$ by~\cite[Table~$8.6$]{bray}). Thus $H=\langle K,x\rangle$ and $t(H,K)=1$.

It remains to consider the case that $p^{6f}-1$ does not admit a large primitive prime divisor. Lemma~\ref{feit} yields $(f,p)\in \{(1,5),(1,3),(2,2),(3,2)\}$. Here the proof follows with the invaluable help of the computer algebra system \texttt{magma}~\cite{magma}.
\end{proof}

\begin{lemma}\label{technical2}
Let $q$ be a prime power and let $H$ be an almost simple group with socle $S:=\mathrm{PSL}_3(q)$ and with $S<H\nleq \mathrm{P}\Gamma\mathrm{L}_3(q)$. Then $t(H)=1$.
\end{lemma}
\begin{proof}
As $\mathrm{PSL}_2(7)\cong \mathrm{PSL}_3(2)$, from Lemma~\ref{technical1}, we may suppose that $q>2$. Here the argument is similar to the proof of Lemma~\ref{technical1}: we use primitive prime divisors and the structure of the subgroup lattice of $S$, see~\cite[Tables~$8.3$,~$8.4$]{bray}. Write $q:=p^f$, where $p$ is a prime number and $f$ is a positive integer. As $q>2$, we have $(p,f)\ne (2,1)$.

Let $K\le H$ with $H=KS$ and let $\theta\in K\setminus S$. From Lemma~\ref{feit}, $p^{3f}-1$ has a large primitive prime divisors, except when $(p,f)\in \{(2,2),(2,4),(2,6),(3,2),(5,2)\}$. For these exceptional cases, we have checked the veracity of this lemma with a computer computation. In particular, for the rest of the argument, we let $r$ be a large primitive prime divisor of $p^{3f}-1$.

A Sylow $r$-subgroup of $S$ is cyclic; let $x\in S$ be an element generating a Sylow $r$-subgroup of $S$. Let $M\in \mathcal{M}(H,x)$.  Here we use the information in~\cite[Tables~$8.3$,~$8.4$]{bray}. From the list of the maximal subgroups of $H$ and recalling that $S<H\nleq \mathrm{P}\Gamma\mathrm{L}_3(q)$ and $r$ is a large primitive prime divisor, we deduce that either $M=\nor H{\langle x\rangle}$, or $f$ is even, $q=q_0^2$ and $M\cap S\cong \mathrm{SU}_3(q_0)$ (here we are using the fact that $r$ is a large Zsigmondy prime and hence $\langle x,x^\theta\rangle$ cannot be contained in a maximal subgroup in the Aschbacher class $\mathcal{S}$ by~\cite[Table~$8.4$]{bray}). In particular, when $f$ is odd, we have 
$\mathcal{M}(H,x)=\{\nor H{\langle x\rangle}\}$. Therefore, we deduce $$\sum_{M\in \mathcal{M}(H,x)}\mu(\theta,M\backslash H)= \mu(\theta,\nor{H}{\langle x\rangle}\setminus H)<1,$$ and hence $t(H,K)=1$, from Lemma~\ref{l:spread}.

Suppose now that $f$ is even and let $\bar{M}\in\mathcal{M}(H,x)\setminus\{\nor H{\langle x\rangle}\}$.  Then $\bar{M}\cap S\cong \mathrm{SU}_3(q_0)$, where $q=q_0^2=p^{f/2}$. Observe that from the ``$c$'' column in~\cite[Table~$8.42$]{bray}, we deduce that the maximal subgroups of $H$  with $\bar{M}\cap S$ isomorphic to $\mathrm{SU}_3(q_0)$ form $\gcd(q_0-1,3)$ $S$-conjugacy class.  Let $\Omega_1:=\{\langle x^g\rangle\mid g\in H\}$.
Using the information in~\cite[Table~$8.3$]{bray}, we deduce
$$|\Omega_1|=\frac{q^3(q^3-1)(q^2-1)}{(q^2+q+1)3}=\frac{q^3(q^2-1)(q-1)}{3}.$$

Let $\Omega_2:=\{\bar{M}^g\mid g\in H\}$. Using the information in~\cite[Table~$8.3$]{bray}, we deduce
$$|\Omega_2|=\frac{q^3(q^3-1)(q^2-1)}{(q_0^3+1)q_0^3(q_0^2-1)}=q_0^3(q_0^3-1)(q_0^2+1).$$
How, consider the bipartite graph having vertex set $\Omega_1\cup\Omega_2$ and having edge set consisting of the pairs $\{A,B\}$ with $A\in \Omega_1$, $B\in \Omega_2$ and $A\le B$. Fix $B\in \Omega_2$. Using the structure of the unitary  group $B$, we see that the number of $A\in \Omega_1$ with $A\le B$ is
$$\frac{(q_0^3+1)q_0^3(q_0^2-1)}{(q_0^2-q_0+1)3}=\frac{q_0^3(q_0^2-1)(q_0+1)}{3}.$$
In particular, the number of edges of the bipartite graph is
$$|\Omega_2|\frac{q_0^3(q_0^2-1)(q_0+1)}{3}
=\frac{q^3(q^2-1)(q_0^3-1)(q_0+1)}{3}.$$
This shows that the number of elements in $\Omega_2$ containing the element $\bar{M}\in \Omega_1$ is
$$\frac{\frac{q^3(q^2-1)(q_0^3-1)(q_0+1)}{3}}{|\Omega_1|}=q_0^2+q_0+1.$$
Thus
$$|\mathcal{M}(H,x)|=|\{\nor H{\langle x\rangle}\}\cup\{M\in \Omega_2\mid x\in M\}|=q_0^2+q_0+2.$$

From~\cite[Lemma~$2.10$~(ii)]{burness2}, we have $\mu(\theta,M\backslash H)\le \gcd(3,q-1)/(q_0(q+1))$ for every $M\in \mathcal{M}(\theta,M\setminus H)$ with $M\cap S\cong \mathrm{SU}_3(q_0)$. Moreover, from ~\cite[Theorem~1]{LS}, we have $\mu(\theta,\nor H{\langle x\rangle}\backslash H)\le 4/(3q)$. 
 Therefore
$$\sum_{M\in \mathcal{M}(H,x)}\mu(\theta,M\backslash H)\le\gcd(3,q-1)\frac{q_0^2+q_0+1}{q_0(q+1)}+\frac{4}{3q}<1,$$
whenever $q\notin\{4,16\}$. Since we have excluded the case $q=4$ above, it remains to deal with $q=16$. This case, yet again, has been dealt with a computer computation.
 Now Lemma~\ref{l:spread} shows that $t(H)=1$.
\end{proof}

\begin{lemma}\label{technical3}Let $e$ be a positive integer, let 
$q=3^{2e+1}$
 and let $H$ be an almost simple group with socle 
$S:={}^2G_2(q)$
 and with $H\ne S$. Then $t(H)=1$.
\end{lemma}
\begin{proof}
Let $K\le H$ with $H=KS$ and let $\theta\in K\setminus S$.  Let $r$ be a primitive prime divisor of $q^6-1$. From the structure of the Ree groups ${}^2G_2(q)$, we deduce that the Sylow $r$-subgroups of $S$ are cyclic. Let $x\in S$ be an element generating a Sylow $r$-subgroup of $S$. Using the list of the maximal subgroups of $S$~\cite[Tables~$8.43$]{bray}, we deduce that $|\mathcal{M}(H,x)|=1$. Indeed, $\mathcal{M}(H,x)=\{\nor H{\langle x\rangle}\}$. From~\eqref{eq:spread}, we have $P(\theta,x)\le \mu(\theta,\nor H{\langle x\rangle}\backslash H)<1$. Now Lemma~\ref{l:spread} shows that $t(H)=1$.
\end{proof}

\begin{lemma}\label{technical4}Let $e$ be a positive integer, let 
$q=2^{2e+1}$ and let $H$ be an almost simple group with socle 
$S:={}^2B_2(q)$ and with $H\ne S$. Then $t(H)=1$.
\end{lemma}
\begin{proof}
Let $K\le H$ with $H=KS$ and let $\theta\in K\setminus S$.  Let $r$ be a primitive prime divisor of $q^4-1$. From the structure of the Suzuki groups ${}^2B_2(q)$, we deduce that the Sylow $r$-subgroups of $S$ are cyclic. Let $x\in S$ be an element generating a Sylow $r$-subgroup of $S$. Using the list of the maximal subgroups of $S$~\cite[Tables~$8.16$]{bray}, we deduce that $|\mathcal{M}(H,x)|=1$ and $\mathcal{M}(H,x)=\{\nor H{\langle x\rangle}\}$. Now, the proof follows as in the proof of Lemma~\ref{technical3}.
\end{proof}

\begin{lemma}\label{technical5}Let $e$ be a positive integer with $e\ge 1$, let
 $q=3^{e}$ and let $H$ be an almost simple group with socle $S:=G_2(q)$ and with $H$ containing an outer automorphism which is not a field automorphism. Then $t(H)=1$.
\end{lemma}
\begin{proof}
Recall that $|\mathrm{Aut}(S):S|=2e$. When $e=1$, we have checked the veracity of this lemma with the computer algebra system \texttt{magma}~\cite{magma}. Therefore for the rest of the argument we suppose $e\ge 2$.

Let $K\le H$ with $H=KS$ and let $\theta\in K\setminus S$.  Let $r$ be a primitive prime divisor of $q^6-1$. From the structure of the Lie  group $G_2(q)$, we deduce that the Sylow $r$-subgroups of $S$ are cyclic. Let $x\in S$ be an element generating a Sylow $r$-subgroup of $S$. Let $M\in \mathcal{M}(H,x)$.  Here we use the information in~\cite[Table~$8.42$]{bray}. From the list of the maximal subgroups of $H$ and recalling that $H$ does contain an outer automorphism which is not a field automorphism, we deduce that either $M=\nor H{\langle x\rangle}$, or $e$ is odd and $M\cap S\cong {}^2G_2(q)$ (here we are assuming $e\ge 2$).  In particular, when $e$ is even, we have 
$\mathcal{M}(H,x)=\{\nor H{\langle x\rangle}\}$. Therefore, we deduce $$\sum_{M\in \mathcal{M}(H,x)}\mu(\theta,M\backslash H)= \mu(\theta,\nor{H}{\langle x\rangle}\backslash H)<1,$$ and hence $t(H,K)=1$, from Lemma~\ref{l:spread}.

Suppose now that $e$ is odd and let $\bar{M}\in\mathcal{M}(H,x)\setminus\{\nor H{\langle x\rangle}\}$.  Then $\bar{M}\cap S\cong {}^2G_2(q)$. Observe that from the ``$c$'' column in~\cite[Table~$8.42$]{bray}, we deduce that the maximal subgroups of $H$  with $\bar{M}\cap S$ isomorphic to ${}^2G_2(q)$ form a unique conjugacy class. Observe that 
$$q^6-1=(q^3-1)(q+1)(q+\sqrt{3q}+1)(q-\sqrt{3q}+1).$$
In particular, the primitive prime divisor $r$ of $q^6-1$ can be chosen so that $r$ divides $q+\sqrt{3q}+1$.
Let $\Omega_1:=\{\langle x^g\rangle\mid g\in H\}$. Using the information in~\cite[Table~$8.42$]{bray}, we deduce
$$|\Omega_1|=\frac{q^6(q^6-1)(q^2-1)}{(q^2-q+1)6}=\frac{q^6(q^3-1)(q^2-1)(q+1)}{6}.$$

Let $\Omega_2:=\{\bar{M}^g\mid g\in H\}$. Using the information in~\cite[Table~$8.42$]{bray}, we deduce
$$|\Omega_2|=\frac{q^6(q^6-1)(q^2-1)}{(q^3+1)q^3(q-1)}=q^3(q^3-1)(q+1).$$
Now, consider the bipartite graph having vertex set $\Omega_1\cup\Omega_2$ and having edge set consisting of the pairs $\{A,B\}$ with $A\in \Omega_1$, $B\in \Omega_2$ and $A\le B$. Fix $B\in \Omega_2$. Using the structure of the Ree group $B$, we see that the number of $A\in \Omega_1$ with $A\le B$ is
$$\frac{(q^3+1)q^3(q-1)}{(q+\sqrt{3q}+1)6}=\frac{(q-\sqrt{3q}+1)q^3(q^2-1)}{6}.$$
In particular, the number of edges of the bipartite graph is
$$|\Omega_2|\frac{(q-\sqrt{3q}+1)q^3(q^2-1)}{6}
=\frac{q^6(q^3-1)(q^2-1)(q-\sqrt{3q}+1)(q+1)}{6}.$$
This shows that the number of elements in $\Omega_2$ containing the element $\bar{M}\in \Omega_1$ is
$$\frac{\frac{q^6(q^3-1)(q^2-1)(q-\sqrt{3q}+1)(q+1)}{6}}{|\Omega_1|}=q-\sqrt{3q}+1.$$
Thus
$$|\mathcal{M}(H,x)|=|\{\nor H{\langle x\rangle}\}\cup\{M\in \Omega_2\mid x\in M\}|=q-\sqrt{3q}+2.$$

From~\cite[Theorem~1]{LLS}, we have $\mu(\theta,M\backslash H)<1/(q^2-q+1)$ for every $M\in \mathcal{M}(\theta,M\setminus H)$. Therefore
$$\sum_{M\in \mathcal{M}(H,x)}\mu(\theta,M\backslash H)\le\frac{q-\sqrt{3q}+2}{q^2-q+1}<1.$$
 Now Lemma~\ref{l:spread} shows that $t(H)=1$.
\end{proof}

\begin{lemma}\label{technical6}Let $e$ be a positive integer with $e\ge 2$, let 
$q=2^{e}$ and let $H$ be an almost simple group with socle $S:=\mathrm{Sp}_4(q)$ and with $H$ containing an outer automorphism which is not a field automorphism. Then $t(H)=1$.
\end{lemma}
\begin{proof}
Recall that $|\mathrm{Aut}(S):S|=2e$. Let $K\le H$ with $H=KS$ and let $\theta\in K\setminus S$.  Let $r$ be a primitive prime divisor of $q^4-1$. From the structure of the classical  group $\mathrm{Sp}_4(q)$, we deduce that the Sylow $r$-subgroups of $S$ are cyclic. Let $x\in S$ be an element generating a Sylow $r$-subgroup of $S$. 

Let $M\in\mathcal{M}(H,x)$. Here we use the information in~\cite[Table~$8.14$]{bray}. From the list of the maximal subgroups of $H$ and recalling that $H$ does contain an outer automorphism which is not a field automorphism, we deduce that either $M=\nor H{\langle x\rangle}$, or $e$ is odd and $M\cap S\cong {}^2B_2(q)$.  In particular, when $e$ is even, we have 
$\mathcal{M}(H,x)=\{\nor H{\langle x\rangle}\}$. Therefore, we deduce $$\sum_{M\in \mathcal{M}(H,x)}\mu(\theta,M\backslash H)= \mu(\theta,\nor{H}{\langle x\rangle}\backslash H)<1,$$ and hence $t(H,K)=1$, from Lemma~\ref{l:spread}.

Suppose now that $e$ is odd and let $\bar{M}\in\mathcal{M}(H,x)\setminus\{\nor H{\langle x\rangle}\}$. Then $\bar{M}\cap S\cong {}^2B_2(q)$. Observe that from the ``$c$'' column in~\cite[Table~$8.14$]{bray}, we deduce that the maximal subgroups of $H$  with $\bar{M}\cap S$ isomorphic to ${}^2B_2(q)$ form a unique conjugacy class. Observe that 
$$q^4-1=(q^2-1)(q+\sqrt{2q}+1)(q-\sqrt{2q}+1).$$
In particular, the primitive prime divisor $r$ of $q^4-1$ can be chosen so that $r$ divides $q+\sqrt{2q}+1$.
Let $\Omega_1:=\{\langle x^g\rangle\mid g\in H\}$. Using the information in~\cite[Table~$8.14$]{bray}, we deduce
$$|\Omega_1|=\frac{q^4(q^4-1)(q^2-1)}{(q^2+1)4}=\frac{q^4(q^2-1)^2}{4}.$$
Let $\Omega_2:=\{\bar{M}^g\mid g\in H\}$. Using the information in~\cite[Table~$8.14$]{bray}, we deduce
$$|\Omega_2|=\frac{q^4(q^4-1)(q^2-1)}{(q^2+1)q^2(q-1)}=q^2(q^2-1)(q+1).$$
How, consider the bipartite graph having vertex set $\Omega_1\cup\Omega_2$ and having edge set consisting of the pairs $\{A,B\}$ with $A\in \Omega_1$, $B\in \Omega_2$ and $A\le B$. Fix $B\in \Omega_2$. Using the structure of the Suzuki group $B$, we see that the number of $A\in \Omega_1$ with $A\le B$ is
$$\frac{(q^2+1)q^2(q-1)}{(q+\sqrt{2q}+1)4}=\frac{(q-\sqrt{2q}+1)q^2(q-1)}{4}.$$ In particular, the number of edges of the bipartite graph is
$$|\Omega_2|\frac{(q-\sqrt{2q}+1)q^2(q-1)}{4}=\frac{q^4(q^2-1)^2(q-\sqrt{2q}+1)}{4}.$$
This shows that the number of elements in $\Omega_2$ containing the element $\bar{M}\in \Omega_1$ is
$$\frac{\frac{q^4(q^2-1)^2(q-\sqrt{2q}+1)}{4}}{|\Omega_1|}=q-\sqrt{2q}+1.$$
Thus
$$|\mathcal{M}(H,x)|=|\{\nor H{\langle x\rangle}\}\cup \{M\in \Omega_2\mid x\in M\}|=q-\sqrt{2q}+2.$$

Now,~\cite[Theorem~1]{burness} yields $\mu(\theta,M\backslash H)\le |\theta^H|^{-\frac{1}{4}}=|H:\cent H\theta|^{-\frac{1}{4}}$ for every  $M\in \mathcal{M}(H,x)$. As $\theta$ is an outer automorphism which is not a field automorphism and as $e$ is odd, replacing $\theta$ with a suitable power, we may suppose that $\theta$ is an involution and that $\theta$ is a graph-field automorphism. From~\cite{GLS}, we deduce that $\cent S\theta\cong {}^2B_2(q)$ and hence $$|\theta^{H}|=\frac{q^4(q^4-1)(q^2-1)}{(q^2+1)q^2(q-1)}=q^2(q^2+1)(q+1).$$ Therefore
$$\sum_{M\in \mathcal{M}(M,x)}\mu(\theta,M\backslash H)\le\frac{q-\sqrt{2q}+2}{(q^2(q^2+1)(q+1))^{1/4}}<1,$$
where the last inequality follows with a computation.
 Now Lemma~\ref{l:spread} shows that $t(H)=1$.
\end{proof}

\begin{lemma}\label{technical}Let $H$ be an almost  simple group with socle $S$. Then there exists a subgroup $K$ of $H$ with $H=KS$ and with $m_K(H)>t(H)$. 
\end{lemma}
\begin{proof}
Suppose first $H=S$. Choose $K:=1$. Then $m_K(H)=m(H)\ge 3$, because we can generate $H=S$ with conjugated involutions. Therefore, the proof follows from~\eqref{eq:silly2}. Thus, for the rest of the argument, we suppose $H\ne S$.
Now, we use the Classification of Finite Simple Groups and we divide our proof depending on the type of $S$.

\smallskip

\noindent\textsc{Alternating groups:} Suppose $S$ is an alternating group $\mathrm{Alt}(n)$ of degree $n\ge 5$. Assume first $n\ne 6$, or $n=6$ and $H=\Sym(6)$. Then $H=\Sym(n)$. Choose $K:=\langle (1,2)\rangle$ and let
$$\Lambda:=\{(1,2,3),  (1,2)(3,4), (1,2)(3,5), \dots, (1,2)(3,n)\}.$$
It is readily seen that $\Lambda$ is a  $K$-independent generating set for $H$. Therefore, $m_K(H)\ge |\Lambda|=n-2\geq 3$ and the proof follows again from~\eqref{eq:silly2}.

As $\mathrm{Alt}(6)\cong \mathrm{PSL}_2(9)$, we postpone the proof of the case $n=6$ and $H\ne \mathrm{Sym}(6)$, when we deal with groups of Lie type.

\smallskip

\noindent\textsc{Sporadic groups:} Suppose $S$ is a sporadic simple group. As $H\ne S$, we deduce $H=\aut S$ and $S$ is one of the following groups
$$Fi_{22}, Fi_{24}, HN, J_3, M_{22}, O'N, HS, J_2, McL, He, M_{12}, Suz.$$

If $S\in \{Fi_{22}, Fi_{24}, HN, J_3, M_{22}, O'N\}$, then it follows from~\cite[Table 9]{bgk} that $t(H)=1$.  However, if we choose $\alpha$ an involution from $H\setminus S$ and we set $K:=\langle \alpha\rangle$, then $m_K(H)\geq 2,$ because we can generated $H$ with $\alpha$ and a suitable number (at least 2) of involutions from $S.$  

If $S \in \{HS, J_2, McL, He, M_{12}, Suz\}$, we have verified that $m_K(H)\geq 3$ using \texttt{magma}: in all cases there exists $\alpha \in H\setminus S$ with $|\alpha|=2$ and three conjugated involutions in $S$ such that
$\{\alpha, t_1, t_2, t_3\}$ is a $\langle\alpha\rangle$-independent generating set for $H.$

\smallskip

\noindent\textsc{Groups of Lie type:} Here we use the information and the notation in~\cite[Section~2.4]{GLS}. The simple group of Lie type $S$ is generated by root elements $x_{\pm\hat{\alpha}}(t)$, where $\alpha\in \Pi$, $\Pi$ is a fundamental system for the root system $\Sigma$ of $S$, and $t$ lies in a suitable finite field $\mathbb{F}$. As $x_{\hat{\alpha}}(t)$ is unipotent, $x_{\hat{\alpha}}(t)$ has prime order and hence it is a $pp$-element. 

The action of the automorphism group of $S$ on the root elements $x_{\pm\hat{\alpha}}(t)$ is described in~\cite[Section~2.5]{GLS} and again we use the information and the notation therein. The outer automorphisms of $S$ are divided in inner-diagonal, field and graph automorphisms. These can be chosen so that inner-diagonal and field automorphisms normalize each root subgroup $\langle x_{\hat{\alpha}}(t)\mid t\in\mathbb{F}\rangle$; whereas, graph automorphisms permute the root subgroups according to the action of the graph automorphism on the nodes of the Dynkin diagram. In particular, we may choose a supplement $K$ of $S$ in $H$ so that the elements in $K$ consist of inner-diagonal, field and graph automorphisms, with respect to the choice of the root system $\Sigma$. Now, let $\tilde{\Pi}\subseteq \Pi$ be a set of representatives of the orbits for the action of $K$ on $\Pi$. Then
$$H=\langle K,x_{\hat{\alpha}}(t)\mid \alpha\in {\pm \tilde\Pi},t\in \mathbb{F}\rangle$$
and hence from the set $\{x_{\hat{\alpha}}(t)\mid \alpha\in \pm\tilde{\Pi},t\in \mathbb{F}\}$ we may extract a $K$-independent generating set $Y$ for $H$ consisting of $pp$-elements. For each $\beta \in \pm \Pi$, define $S_\beta:=\langle x_{\hat{\alpha}}(t)\mid \alpha\in\pm \Pi\setminus\{\beta\}, t\in \mathbb{F} \rangle.$ Observe that $S_\beta$ is contained  in a proper parabolic subgroup of $S$ normalized by $K$. This implies $|Y|\ge 2|\tilde{\Pi}|$. A direct inspection on the various root systems gives that one of the following holds:
\begin{enumerate}
\item\label{ITEM1} $|\tilde{\Pi}|\ge 2 $, or
\item\label{ITEM2} $S$ is a simple group of Lie type of Lie rank $1$, that is, $S\in \{A_1(q)=\mathrm{PSL}_2(q),\,{}^2A_2(q)=\mathrm{PSU}_3(q),\,{}^2B_2(q),\,{}^2G_2(q)\},$ or
\item\label{ITEM3} $S=A_2(q)=\mathrm{PSL}_3(q)$ and $H\nleq \mathrm{P}\Gamma\mathrm{L}_3(q)$, 
\item\label{ITEM4} $S=B_2(q)=\mathrm{PSp}_4(q)$, $q=2^e$ for some $e\ge 1$ and $H$ contains an outer automorphism which is not a field automorphism,
\item\label{ITEM5} $S=G_2(q)$, $q=3^e$ for some $e\ge 1$, and $H$ contains an outer automorphism which is not a field automorphism.
\end{enumerate} 

If~\eqref{ITEM1} holds, then the proof follows from~\eqref{eq:silly2}. In the remaining cases, we have shown in Lemmas~\ref{technical1},~\ref{technical2U},~\ref{technical2},~\ref{technical3},~\ref{technical4},~\ref{technical5} and~\ref{technical6} that $t(H)=1$. Using this slight refinement on the value of $t(H)$ and repeating the argument above for the remaining groups we deduce $m_K(H)\ge 2>1=t(H)$.
\end{proof}

\subsection{Pulling the threads of the argument}
\begin{proof}[Proof of Theorem~$\ref{thrm:main}$] We argue by contradiction and among all non-soluble $\mathcal B_{pp}$-groups we choose $G$ having minimal order. 

Let $N$ be a minimal normal subgroup of $G$. From Lemma~\ref{induction}, $G/N$ is a  $\mathcal{B}_{pp}$-group and hence, from our minimal choice of $G$, we deduce that 
\begin{equation}\label{eq:silly1}
G/N \hbox{ is solvable}.
\end{equation}

Suppose that $G$ has two distinct minimal normal subgroups $N_1$ and $N_2$. Since $N_1\cap N_2=1$, $G$ embeds into the cartesian product $G/N_1\times G/N_2$. As $G/N_1$ and $G/N_2$ are both solvable, we deduce that $G$ is solvable, which is a contradiction. Therefore, $G$ has a unique minimal normal subgroup $N$, that is, $G$ is monolithic. 

If $N$ is abelian, then $G$ is solvable by~\eqref{eq:silly1}, which is a contradiction. Therefore, $N$ is non-abelian and hence $N\cong S^n$, for some non-abelian simple group $S$. Write $N:=S_1\times \cdots \times S_n$, where $S_1,\ldots,S_n$ are the simple direct factors of $N$. Let $H$ be the subgroup of $\aut(S)$ induced by the conjugacy action of $\nor G{S_1}$ on $S$. Clearly, $H$ is an almost simple group with socle $S$. Moreover, since $G$ is monolithic, $G$ embeds into the wreath product $H \wr \perm(n)$ and hence, without loss of generality, we may assume that $G$ is a subgroup of $H\,\mathrm{wr} \Sym(n)$ with $S^n\le G$ and with $$\pi:\nor G{S_1}\to H$$ projecting onto $H$. In particular, we may write the elements of $G$ as ordered pairs $f\sigma$, with $f\in H^n$ and $\sigma\in \Sym(n)$.

Let $$m_1=m(G/N).$$ 
Let $$Y=\{g_1,\dots,g_{m_1}\}$$ be a set of $pp$-elements of $G$ with $\{g_1N,\ldots,g_{m_1}N\}$ a $pp$-base for $G/N$. 

Let
$$K:=\pi(\nor{\langle Y\rangle}{S_1}).$$
As $G=\langle Y\rangle N$, from the modular law  we get 
$$\nor G{S_1}=\nor G{S_1}\cap G=(\nor G{S_1}\cap \langle Y\rangle)N=\nor{\langle Y\rangle} {S_1}N.$$ 
Thus $$H=\pi(\nor {G}{S_1})=\pi(\nor {\langle Y\rangle}{S_1})\pi(N)=KS.$$

Let $X$  be a set of $pp$-elements in $S$ with $H=\langle X, K\rangle$ and having cardinality $t(H,K).$ Let
$$\tilde X:=\{(x,\underbrace{1,\dots,1}_{n-1 \hbox{ times}})\in N \mid x \in X\}$$
and observe that $\tilde X\subseteq S^n=N\le G\le H\,\mathrm{wr}\Sym(n).$

As $N$ is a minimal normal subgroup of $G$, $G$ acts transitively by conjugation on the set $\{S_1,\ldots,S_n\}$ of simple direct factors of $N$. From this, it follows that $Y \cup \tilde X$ is a generating set for $G$. As $Y\cup \tilde{X}$ consists  of $pp$-elements and as all $pp$-bases of $G$ have the same cardinality, we get $m_{pp}(G)\le m_1+t(H,K)\le m_1+t(H)$. Thus 
\begin{equation}\label{eq:22}
m(G)\le m_1+t(H),
\end{equation} by Lemma~\ref{l:1}.

Recall the definition of $\mu(G)$ and $\mu(S)$ in Section~\ref{notation}. In~\cite[page~403, inequality~(1)]{l2} and in~\cite[Proposition~4]{l1}, it is proved that $\mu(G)\geq \mu(H)$. Moreover, by~\cite[Lemma~7]{l1}, we have $\mu(H)\geq m_K(H)$, for every subgroup $K$ of $H$ with $H=KS$. In particular, combining these two results, we deduce $\mu(G)\ge m_K(H)$. From~\eqref{eq:22}, we get
$$t(H)\ge m(G)-m_1=m(G)-m(G/N)=\mu(G)\ge m_K(H),$$
for every subgroup $K$ of $H$ with $H=KS$. However, this contradicts Lemma~\ref{technical}.
\end{proof}

\begin{proof}[Proof of Corollary~$\ref{corollary}$]Let $G$ be a $\mathcal{B}_{pp}$-group with $\Phi(G)=1$. From Theorem~\ref{thrm:main}, $G$ is solvable and hence the proof now follows from~\cite[Theorem~$1.2$]{s}.
\end{proof}

\section{Proof of Theorem~$\ref{bound}$}\label{sec:3}
Let $G$ be a finite group. Take a chief series
$$1=G_t \unlhd \dots \unlhd G_0=G$$ and consider the non-negative integers $\mu_i=m(G/G_{i+1})-m(G/G_i).$
Clearly \begin{equation}\label{sommamu} m(G)=\sum_{0\leq i \leq t-1}\mu_i.
\end{equation}
Information on the values of $\mu_i$ have been obtained in \cite {l1}, where is it proved in particular:
\begin{itemize}
	\item if $G_i/G_{i+1}$ is abelian, then $\mu_i=0$ if $G_{i+1}/G_{i}\leq \Phi(G/G_{i+1})$, $\mu_i=1$ otherwise;
		\item if $G_i/G_{i+1}$ is non-abelian, then $\mu_i=\mu_i(L_i)=m(L_i)-m(L_i/\soc L_i)$, where $L_i=G/C_G(G_i/G_{i+1}).$
\end{itemize}
In the second case, $L_i$ is a monolithic group and $\soc L_i=S_i^{n_i}$ where $n_i$ is a positive integer and $S_i$ is a finite non-abelian simple group.  As we already recalled in the previous section, by ~\cite[page~403, inequality~(1)]{l2} and ~\cite[Proposition~4]{l1}, there exists an almost simple group $H_i$ such that $\soc H_i=S_i$ and  $\mu_i=\mu(L_i)\geq \mu(H_i)$. Moreover, by~\cite[Lemma~7]{l1}, we have $\mu(H_i)\geq m_{K_i}(H_i)$, for every subgroup $K_i$ of $H_i$ with $H_i=K_iS_i$. By the results in Section \ref{proofs}, for every choice of $H_i$ there exists $K_i$ such that $K_iS_i=H_i$ and $m_{K_i}(H_i)\geq 2.$  So $\mu_i\geq 2$ whenever $G_i/G_{i+1}$ is non-abelian, and therefore the statement of Theorem \ref{bound} follows 
from~\eqref{sommamu}.

\end{document}